\documentclass[12pt]{amsart}

\usepackage{amssymb}

\usepackage{enumitem}

\usepackage{graphicx}

\makeatletter
\@namedef{subjclassname@2020}{%
  \textup{2020} Mathematics Subject Classification}
\makeatother

\usepackage[T1]{fontenc}

\newtheorem{theorem}{Theorem}[section]

\newtheorem{lemma}[theorem]{Lemma}

\theoremstyle{definition}
\newtheorem{definition}[theorem]{Definition}
\newtheorem{remark}[theorem]{Remark}
\newtheorem{example}[theorem]{Example}
\newtheorem{prop}[theorem]{Proposition}

\numberwithin{equation}{section}

\usepackage{graphicx, tikz}

\newcommand{\C}{{\mathbb C}}
\newcommand{\D}{{\mathbb D}}
\newcommand{\E}{{\mathbb E}}
\newcommand{\K}{{\mathbb K}}
\newcommand{\N}{{\mathbb N}}
\newcommand{\Q}{{\mathbb Q}}
\newcommand{\R}{{\mathbb R}}
\newcommand{\T}{{\mathbb T}}
\newcommand{\Z}{{\mathbb Z}}
\renewcommand{\S}{{\mathbb S}}

\newcommand\cA{{\mathcal A}}
\newcommand\cB{{\mathcal B}}
\newcommand\cC{{\mathcal C}}
\newcommand\cD{{\mathcal D}}
\newcommand\cE{{\mathcal E}}
\newcommand\cF{{\mathcal F}}
\newcommand\cG{{\mathcal G}}
\newcommand\cH{{\mathcal H}}
\newcommand\cI{{\mathcal I}}
\newcommand\cK{{\mathcal K}}
\newcommand\cL{{\mathcal L}}
\newcommand\cO{{\mathcal O}}
\newcommand\cM{{\mathcal M}}
\newcommand\cN{{\mathcal N}}
\newcommand\cT{{\mathcal T}}
\newcommand\cW{{\mathcal W}}
\newcommand\cY{{\mathcal Y}}
\newcommand\cP{{\mathcal P}}
\newcommand\cR{{\mathcal R}}
\newcommand\cS{{\mathcal S}}

\newcommand\bA{{\mathbb A}}
\newcommand\bB{{\mathbb B}}
\newcommand\bC{{\mathbb C}}
\newcommand\bD{{\mathbb D}}
\newcommand\bE{{\mathbb E}}
\newcommand\bH{{\mathbb H}}
\newcommand\bG{{\mathbb G}}
\newcommand\bN{{\mathbb N}}
\newcommand\bP{{\mathbb P}}
\newcommand\bQ{{\mathbb Q}}
\newcommand\bR{{\mathbb R}}
\newcommand\bS{{\mathbb S}}
\newcommand\bT{{\mathbb T}}
\newcommand\bV{{\mathbb V}}
\newcommand\bZ{{\mathbb Z}}
\newcommand\Id{{\bf 1}}
\newcommand\Alf{{\mathcal A}}
\newcommand\void{{\varnothing}}
\newcommand{\bg}{\boldsymbol{\gamma}}
\newcommand{\sbv}{\boldsymbol{s}} 
\newcommand{\tbv}{\boldsymbol{t}}  
\newcommand{\sg}{\boldsymbol{s}} 
\newcommand{\tg}{\boldsymbol{t}}  
\newcommand{\ld}[1]{{}^{\dagger}#1}  
\newcommand{\rd}[1]{#1^{\dagger}}  
\newcommand{\IL}{\underleftarrow\lim}  
\newcommand{\BV}{\text{\tiny BV}}
\newcommand{\orb}{\operatorname{orb}}

\def\err{\mbox{err}}
\def\itin{\boldsymbol{i}}
\newcommand{\tb}{|\!|\!|}
\newcommand{\1}{\boldsymbol{1}}

\newcommand\ve{\varepsilon}
\newcommand\eps{\epsilon}
\newcommand\vf{\varphi}
\def\eps{\varepsilon}
\def\ph{\varphi}
\def\GC{\boldsymbol{G}}
\def\ie{{\em i.e.,} }

\begin{document}


\baselineskip=17pt

	\title{Translation algorithms for graph covers}
	\author{Jan Boro\'nski, Henk Bruin, Przemys{\l}aw Kucharski}
	
	\date{\today}

	\begin{abstract}
		Graph covers are a way to describe continuous maps (and homeomorphisms) of a Cantor set, more generally than e.g.\ Bratteli-Vershik systems.
		Every continuous map on a zero-dimensional compact set can be expressed by a graph cover (e.g.\ non-minimality or aperiodicty are no restrictions).
		We give a survey on the construction, properties and some special cases of graph covers.
	\end{abstract}



\maketitle

%

	\bigskip
	
	Graph covers were introduced by Gambaudo \& Martens \cite{GM06} as a general way to describe minimal Cantor systems, and they used this, among other things, to construct Cantor systems with unusual Choquet simplices of invariant measures.
	Akin, Glasner \& Weiss \cite{AGW08} used a similar approach to construct Cantor systems whose conjugacy class is a dense $G_\delta$-set within the class of all Cantor systems.
	Finding such universal systems was also the motivation of Shimomura
	\cite{Shi14} to study graph covers, and in a series of papers, he established many properties of graph covers, including that all Cantor systems (whether minimal, distal, aperiodic or none of these)
	can be represented as a graph cover. In particular, he proved \cite[Theorem 3.9]{Shi14}:
	
	\begin{theorem}[Shimomura]
		Every continuous map on a zero-dimensional space is conjugate to some graph cover.
	\end{theorem}
	
	The technique is used in \cite[Theorem 1.1]{Shi20} to extend results of Herman et al.\ \cite{HPS92} and Medynets \cite{Med06} to: every homeomorphism on a zero-dimensional set has a representation as a Bratteli-Vershik system. A variant of this method was also used by Good \& Meddaugh \cite{GoodM} to give a characterization of shadowing in terms of inverse limits of shifts of finite type satisfying the Mittag-Leffler condition. Other applications include characterization of Cantor systems that can be embedded into real line with vanishing derivative everywhere \cite{BKO2,BKO3,GO}, the construction of completely scrabled systems with transitivity \cite{Shi16PAMS}, and mixing \cite{BKO3}, and almost minimal systems \cite{FGP}. Related graph theoretic approach was used in \cite{BD} by Bernardes \& Darji to characterize when two Cantor systems are conjugate.
	An application of graph covers to show that particular infinitely renormalizable
	Lorenz maps on the interval are not uniquely ergodic on their Cantor attractor was given by Martens \& Winckler \cite{MW18}.
	A generalization of the method to higher dimensional systems was proposed in \cite{Ku}.

	We aim to give a brief overview of the concept of graph covers and prove some relations to other constructions of (minimal) Cantor systems.
	We present our own proofs, bypassing some of the constructions of Shimomura. We include questions in some of the sections. 
	
	\section{Graph Covers}
	
	Let $(\GC_n)_{n \geq 0}$ be a sequence of directed graphs
	(or rather edge-sets of directed graphs).
	We call the directed edges $\gamma_n \in \GC_n$ arrows.
	They connect the vertices $\gamma_n = (v \to v')$;
	we write $v = \sg(\gamma_n)$ and $v' = \tg(\gamma_n)$
	for the source and target of the arrow.
	We stipulate that every vertex has at least one outgoing and at least one incoming arrow.
	Therefore $\GC_n$ has no end-points, only regular vertices (with exactly one outgoing and one incoming arrow) and branch-points.
	The graph $\GC_0$ consists of a single vertex $\epsilon$ from which a finite number of directed loops $\gamma:(\epsilon\to\epsilon)$ emerge.
	
	A graph cover is the inverse limit space of directed graphs $\GC_n$:
	$$
	\GC = \IL(\GC_n,\pi_n)
	= \{ (\gamma_n)_{n \geq 0} : \pi_{n+1}(\gamma_{n+1}) = \gamma_n \in \GC_n \text{ for all } n \geq 0\},
	$$
	where the $\pi_n$'s are called the bonding maps.

	For each $\gamma \in \GC_n$,
	$\pi_n(\gamma)$ is a single arrow in $\GC_{n-1}$ such that
	if $\sg(\gamma) = \tg(\gamma')$, then  $\sg(\pi_n(\gamma)) = \tg(\pi_n(\gamma'))$.
	In particular, $\pi_n$ preserves the direction of the arrows.
	
	We stipulate the following properties of the graphs and bonding maps:
	\\
	(i)  the $\pi_n$'s are {\bf edge surjective}:
	\begin{quote}
		For each arrow $\gamma' \in \GC_{n-1}$, there is an arrow $\gamma \in \GC_n$ such that $\pi_n(\gamma)=\gamma'$.
	\end{quote}
	(ii) the $\pi_n$'s are {\bf positive directional}:
	\begin{quote}
		If $\gamma$ and $\gamma'$ are two arrows in $\GC_n$ starting at the same vertex $\sg(\gamma) = \sg(\gamma')$,
		then $\tg(\pi_n(\gamma)) = \tg(\pi_n(\gamma'))$ is the same vertex in $\GC_{n-1}$.
	\end{quote}
	Equipped with product topology, $\GC$ is a compact zero-dimensional set.
	We can define a map $f:\GC\to\GC$ by ``moving one step''
	on the vertex sets along the arrows:
	If $\bg = (\gamma_n)_{n \in \N} \in \GC$, then
	\begin{equation}\label{eq:follow-arrow}
	f(\gamma)_n = \gamma'_n \qquad \text{ provided that } \tg(\gamma_n) = \sg(\gamma'_n).
	\end{equation}
	For a single $n$ this definition is ambiguous, but in connection with the bonding maps (namely, that the concatenation $\gamma_n\gamma'_n$
	must lie in the image $\pi_{n+1} \circ \cdots \circ \pi_{n+m}(\gamma_{n+m}\gamma'_{n+m})$ for $m$ sufficiently large) this ambiguity is resolved.
	
	Not part of the required properties, but
	a graph cover is called {\bf negative directional} if the following holds for each $n \in \N$:
	\begin{quote}
		If $\gamma_n$ and $\gamma'$ are two arrows in $\GC_n$ ending at the same vertex $\tg(\gamma) = \tg(\gamma')$,
		then $\sg(\pi_n(\gamma)) = \sg(\pi_n(\gamma'))$ is the same vertex in $\GC_{n-1}$.
	\end{quote}
	A graph that is both positive and negative directional is called {\bf bi-directional}.
	The following result is \cite[Theorem 3.9]{Shi14}.
	
	\begin{theorem}[Shimomura]\label{thm:graph_cover}
		The graph cover $(\GC,f)$ is a continuous map on a compact zero-dimensional space.
		If  $(\GC,f)$ is bi-directional, then $f$ is a homeomorphism.
	\end{theorem}
	
	In order to agree with other graph covering constructions, we can speed up this process by ``telescoping'' between $n$'s where there is a bi-special word.
	This point of view, which holds for shifts in general and in fact for all continuous Cantor systems,
	was introduced by Gambaudo \& Martens \cite{GM06} and studied by several authors, especially Shimomura,
	see e.g.\ \cite{Shi14,Shi16,Shi20}.

	\subsection{Weighted graph covers}
	We can replace paths $\gamma_1 \dots \gamma_k$ in $\GC_n$
	with only regular interior vertices $\tg(\gamma_i)$, $1 \leq i < k$, by a single arrow $\gamma$ with $\sg(\gamma) = \sg(\gamma_1)$,
	$\tg(\gamma) = \tg(\gamma_k)$ and weight $w(\gamma) = k$.
	The rule of $\pi_n$ for such weighted arrows is
	that $\pi_n(\gamma) = \gamma'_1 \dots \gamma'_{\ell}$ in $\GC_{n-1}$
	is only possible if $w(\gamma) = \sum_{j=1}^{\ell} w(\gamma'_j)$.
	That is, $\pi_n$ distributes the weight of $\gamma$ as it were over
	$\gamma'_1, \dots, \gamma'_{\ell}$.
	The map $f:\GC \to \GC$ is adapted accordingly, by allowing steps of size $1/w(\gamma)$ inside arrows $\gamma$ of weight $w(\gamma)$.
	
	As an example, let $\GC_n$ consist of a single vertex
	with two directed loops $\gamma_n$ and $\gamma'_n$.
	The weights $w(\gamma_0) = w(\gamma'_0) = 1$ and assume that
	$$
	\begin{cases}
	\pi_{1}(\gamma'_{1}) = \gamma_0, \qquad \pi_1(\gamma_1) = \underbrace{\gamma_0 \dots \gamma_0}_{a_1 \text{ \small times }} \\
	\pi_{n+1}(\gamma'_{n+1}) = \gamma_n, \qquad
	\pi(\gamma_{n+1}) = \underbrace{ \gamma_n \dots \gamma_n }_{a_{n+1}\  \text{\small times}} \gamma'_n & \text{ if } n \geq 1.
	\end{cases}
	$$
	for some sequence $(a_n)_{n \geq 1}$ in $\N$. This implies that
	$$
	w(\gamma'_{n+1}) = w(\gamma_n), \qquad w(\gamma_{n+1}) = a_n w(\gamma_n) + w(\gamma_{n-1}),
	$$
	which corresponds to the Ostrowski numeration of rotation (or in fact Sturmian shift) of the angle with continued fraction $\alpha = [0;a_1,a_2,a_3,\dots]$, see Figure~\ref{fig:ostrowski}.
	The weight $w(\gamma)$ are the denominators of the convergents of $\alpha$.

	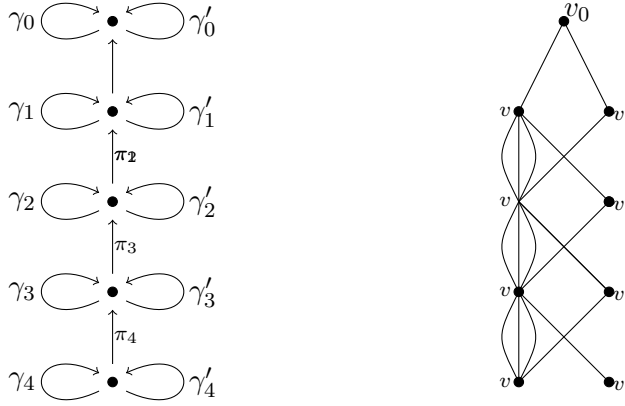
\begin{figure}[ht]
		\begin{center}
			\begin{tikzpicture}[scale=0.6]
			\filldraw(2,10) circle (3pt);
			\draw[->, draw=black] (1.7, 9.8) .. controls (0, 9) and (0,11) .. (1.7, 10.2); \node at (0,10) {\small $\gamma_0$};
			\draw[->, draw=black] (2.3, 9.8) .. controls (4, 9) and (4, 11) .. (2.3, 10.2); \node at (4,10) {\small $\gamma'_0$};
			\draw[->] (2,8.4)--(2,9.6); \node at (2.3,7) {\tiny $\pi_1$};
			\filldraw(2,8) circle (3pt);
			\draw[->, draw=black] (1.7, 7.8) .. controls (0, 7) and (0,9) .. (1.7, 8.2); \node at (0,8) {\small $\gamma_1$};
			\draw[->, draw=black] (2.3, 7.8) .. controls (4, 7) and (4,9) .. (2.3, 8.2); \node at (4,8) {\small $\gamma'_1$};
			\draw[->] (2,6.4)--(2,7.6); \node at (2.3,7) {\tiny $\pi_2$};
			\filldraw(2,6) circle (3pt);
			\draw[->, draw=black] (1.7, 5.8) .. controls (0, 5) and (0,7) .. (1.7, 6.2); \node at (0,6) {\small $\gamma_2$};
			\draw[->, draw=black] (2.3, 5.8) .. controls (4, 5) and (4,7) .. (2.3, 6.2); \node at (4,6) {\small $\gamma'_2$};
			\filldraw(2,4) circle (3pt);
			\draw[->] (2,4.4)--(2,5.6); \node at (2.3,5) {\tiny $\pi_3$};
			\draw[->, draw=black] (1.7, 3.8) .. controls (0, 3) and (0,5) .. (1.7, 4.2); \node at (0,4) {\small $\gamma_3$};
			\draw[->, draw=black] (2.3, 3.8) .. controls (4, 3) and (4,5) .. (2.3, 4.2); \node at (4,4) {\small $\gamma'_3$};
			\filldraw(2,2) circle (3pt);
			\draw[->] (2,2.4)--(2,3.6); \node at (2.3,3) {\tiny $\pi_4$};
			\draw[->, draw=black] (1.7, 1.8) .. controls (0, 1) and (0,3) .. (1.7, 2.2); \node at (0,2) {\small $\gamma_4$};
			\draw[->, draw=black] (2.3, 1.8) .. controls (4, 1) and (4,3) .. (2.3, 2.2); \node at (4,2) {\small $\gamma'_4$};
			\filldraw(12,10) circle (3pt); \node at (12.3,10.2) {\small $v_0$};
			\draw[-] (12,10)--(11,8); \draw[-] (12,10)--(13,8);
			\filldraw(11,8) circle (3pt);  \filldraw(13,8) circle (3pt);
			\node at (10.7, 8) {\tiny $v$}; \node at (13.3,8) {\tiny $v'$};
			\draw[-] (11,6)--(11,8);
			\draw[-, draw=black] (11,6) .. controls (10.5, 7)  .. (11,8);
			\draw[-, draw=black] (11,6) .. controls (11.5, 7)  .. (11,8);
			\draw[-] (11,6)--(13,8); \draw[-] (13,6)--(11,8);
			\filldraw(13,6) circle (3pt);  \filldraw(13,6) circle (3pt);
			\node at (10.7, 6) {\tiny $v$}; \node at (13.3,6) {\tiny $v'$};
			\draw[-] (11,4)--(11,6); \draw[-] (11,4)--(13,6); \draw[-] (13,4)--(11,6);
			\draw[-, draw=black] (11,4) .. controls (10.5, 5)  .. (11,6);
			\draw[-, draw=black] (11,4) .. controls (11.5, 5)  .. (11,6);
			\filldraw(11,4) circle (3pt);  \filldraw(13,4) circle (3pt);
			\node at (10.7, 4) {\tiny $v$}; \node at (13.3, 4) {\tiny $v'$};
			\draw[-] (11,2)--(11,4); \draw[-] (11,2)--(13,4); \draw[-] (13,4)--(11,6);
			\draw[-, draw=black] (11,2) .. controls (10.5, 3)  .. (11,4);
			\draw[-, draw=black] (11,2) .. controls (11.5, 3)  .. (11,4);
			\filldraw(11,2) circle (3pt);  \filldraw(13,2) circle (3pt);
			\draw[-] (13,2)--(11,4);
			\node at (10.7, 2) {\tiny $v$}; \node at (13.3, 2) {\tiny $v'$};
			\end{tikzpicture}
			\caption{\label{fig:ostro} Weighted graph cover (left) and
				ordered Bratteli diagram (right) for the Ostrowski numeration with $a_n \equiv 3$
				and
				$\pi_n(\gamma_n) = { \gamma_{n-1} \cdots \gamma_{n-1} } \gamma'_{n-1}$,
				$\pi_n(\gamma'_n) = \gamma_{n-1}$.
			}
			\label{fig:ostrowski}
		\end{center}
	\end{figure}
	
	The action of $\pi_n$ in a weighted graph cover can be captured in a so-called {\bf winding matrix} $W^n$
	which has size $\# \GC_n \times \# \GC_{n-1}$ and elements
	$W_{\gamma,\gamma'}^n = \#\{ \text{appearances of } \gamma' \text{ in } \pi_n(\gamma)\}$.
	For instance, if the row vector $\vec x = (x_\gamma)_{\gamma \in \GC_n}$ represents a path in $\GC_n$ and
	$x_{\gamma}$ the number of times this path goes through $\gamma$, then
	$\vec x' = (x'_{\gamma'})_{\gamma' \in \GC_{n-1}} = \vec x W^n$
	denotes the number of times the $\pi_n$-images of this path goes through $\gamma'$. Conversely, if $\vec w' = (w'_{\gamma'})_{\gamma' \in \GC_{n-1}}$
	is the column vector of weights of the arrows in $\GC_{n-1}$, then
	$\vec w = (w_{\gamma})_{\gamma \in \GC_n} =  W^n \vec w'$ is the vector of weights of the arrows in $\GC_n$
	
	The winding matrix of a telescoping between $\GC_m$ and $\GC_n$
	for $n \geq m$ is $W = W_n W_{n-1} \cdots W_{m+1}$.
	Winding matrices are the equivalent of what is called incidence matrix or associated matrix in other representations (e.g.\ S-adic or Bratteli-Vershik) of a Cantor system.
	
	\newpage
	
	\subsection{Miscellaneous properties}

	\begin{theorem}
		A graph cover is
		\begin{enumerate}
			\item {\bf chain-transitive} if and only if every $\GC_n$ is
			connected, i.e., between every two $\gamma,\gamma' \in \GC_n$ there is a path with $\gamma$ as first and $\gamma'$ as last edge;

			\item {\bf transitive} if and only if  every $\GC_{n}$ is connected and for every $m$ there exists an $n \geq m$, such that for every loop $\ell_{m}$ in $\GC_m$ there exists a loop $\ell_{n}$ in $\GC_{n}$ with $\ell_m \subset \pi_{m,n}(\ell_{n})$;

			\item {\bf minimal} if and only if for every $m$ there is an $n \geq m$ such that $\GC_m \subset \pi_{m,n}(\ell)$
			for every loop $\ell$ in $\GC_n$;

			\item {\bf uniquely ergodic} if and only if,
			for every $n \in \N$, $\bigcap_{m \geq n}  \R_{\geq 0}^{\# \GC_m}W^{m} \cdot W^{m-1} \cdots W^n$ is a single halfline
			in $\R_{\geq 0}^{\# \GC_n}$.
		\end{enumerate}
	\end{theorem}
	
	\begin{example}
		In the left panel of Figure~\ref{fig:gc_bounded-length},
		$\GC$ has a loop of finite weight, and this precludes minimality.
		Therefore in minimal (weighted) graph covers, the minimal weight of loops has to tend to infinity.
		In the right panel,  of Figure~\ref{fig:gc_bounded-length},
		$\GC$ has an arrow of finite weight, showing that in the theorem above, ``loops'' cannot be replaced by ``arrows''.
		\begin{figure}[ht]
			\begin{center}
				\begin{tikzpicture}[scale=0.7]
				\node[circle, draw=black] at (-1,3) {};
				\draw[->, draw=black] (-0.7, 3.5) .. controls (0, 5) and (-2, 5) .. (-1.3, 3.5);
				\node at (-0.2,4) {\small $a_1$};
				\draw[->, draw=black] (-0.7, 2.5) .. controls (0, 1) and (-2, 1) .. (-1.3, 2.5);
				\node at (-0.2,2) {\small $a_2$};
				\draw[->, draw=black] (2,3) -- (0,3); \node at (1.25,3.3) {\small $\pi_1$};
				\node[circle, draw=black] at (3,3) {};
				\draw[->, draw=black] (3.3, 3.5) .. controls (4, 5) and (2, 5) .. (2.7, 3.5);
				\node at (3.8,4) {\small $a_1$};
				\draw[->, draw=black] (3.3, 2.5) .. controls (4, 1) and (2, 1) .. (2.7, 2.5);
				\node at (3.8,2) {\small $a_2$};
				\draw[->, draw=black] (6,3) -- (4,3); \node at (5.25,3.3) {\small $\pi_2$};
				\node[circle, draw=black] at (10,4) {};
				\draw[->, draw=black] (10.3, 4.5) .. controls (11, 6) and (9, 6) .. (9.7, 4.5);
				\node at (10.8,5) {\small $a_3$};
				\draw[->, draw=black] (10.2, 2.5) .. controls (10.5, 3) .. (10.2, 3.5); \node at (10.8,2.5) {\small $a_2$};
				\draw[->, draw=black] (9.8, 3.5) .. controls (9.5, 3) .. (9.8, 2.5); \node at (9.2,2.5) {\small $a_1$};
				\node at (10.8,5) {\small $a_3$};
				\node[circle, draw=black] at (10,2) {};
				\draw[->, draw=black] (10.3, 1.5) .. controls (11, 0) and (9, 0) .. (9.7, 1.5);
				\node at (10.8,1) {\small $a_4$};
				\draw[->, draw=black] (13,3) -- (11,3); \node at (12.25,3.3) {\small $\pi_1$};
				\node[circle, draw=black] at (14,4) {};
				\draw[->, draw=black] (14.3, 4.5) .. controls (15, 6) and (13, 6) .. (13.7, 4.5);
				\node at (14.8,5) {\small $a_3$};
				\draw[->, draw=black] (14.2, 2.5) .. controls (14.5, 3) .. (14.2, 3.5); \node at (14.8,2.5) {\small $a_2$};
				\draw[->, draw=black] (13.8, 3.5) .. controls (13.5, 3) .. (13.8, 2.5); \node at (13.2,2.5) {\small $a_1$};
				\node at (14.8,5) {\small $a_3$};
				\node[circle, draw=black] at (14,2) {};
				\draw[->, draw=black] (14.3, 1.5) .. controls (15, 0) and (13, 0) .. (13.7, 1.5);
				\node at (14.8,1) {\small $a_4$};
				\draw[->, draw=black] (17,3) -- (15,3); \node at (16.25,3.3) {\small $\pi_2$};
				\node at (2,-2) {$\pi_n: \begin{cases}
					1 \to 1\\
					2 \to 1221
					\end{cases}
					$};
				\node at (13,-2) {$\pi_n: \begin{cases}
					1 \to 1\\
					2 \to 23142\\
					3 \to 1423\\
					4 \to 2314
					\end{cases}
					$};
				\end{tikzpicture}
				\caption{Two examples of graph covers where the length/weight of arrows in $\GC_n$ doesn't tend to infinity.}
				\label{fig:gc_bounded-length}
			\end{center}
		\end{figure}
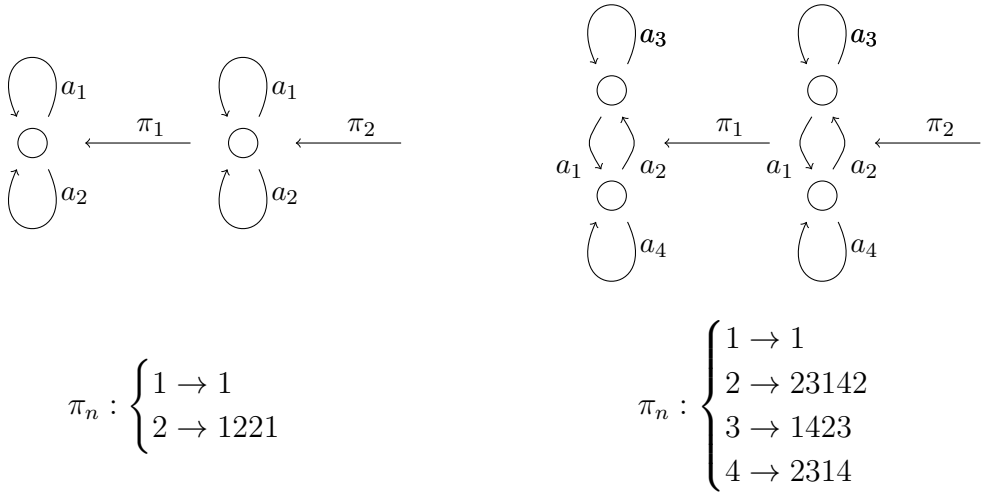
		
	\end{example}

	\begin{proof} 
		\begin{enumerate}
			\item The result on chain-transitivity is  \cite[Proposition 2.8]{Shi16}.
			
			\item $\Rightarrow$: Take $m\in \N$ arbitrary and $n\geq m$. Let $x\in \GC$ be a point with the dense orbit such that $\ell_{m}\subset\{f^{i}(x)_{m}\}_{i=0}^{k-1}$ and $x_{s}=f^{k}(x)_{s}$ for some $k\in\N$ and $s\in\{n,m\}$. Then the path $\ell_n=((f^{i}(x)_{n},f^{i+1}(x)_{n}))_{i=0}^{k-1}\subset \GC_{n}$ is a loop with $\ell_{m}\subset \pi_{m,n}(\ell_{n})$.

			$\Leftarrow$: Apply the hypothesis recursively to obtain a sequence of loops $\ell_n$ with $\ell_n \subset\pi_{n+1}(\ell_{n+1})$ such that $\ell_n$ passes through every edge of $\GC_n$.
			Without loss of generality, assume that $\pi_{n+1}$ maps the first edge $\gamma_{n+1}$ of $\ell_{n+1}$ to the first edge $\gamma_n$ of $\ell_n$. Then $\gamma=(\gamma_n)_{n \geq 0}\in \GC$ is a point with dense orbit in $\GC$.
			
			\item The result on minimality is contained in \cite[Theorem 3.5]{Shi16}; it is analogous to the fact that primitive S-adic shifts are minimal.
			To prove the ``if''-direction, take any neighbourhood $U$ in $\GC$; there is $m$ and a loop in $\ell_m \in \GC_m$ such that
			the cylinder set given by this loop is contained in $U$.
			Next take $(\gamma)_{n\geq 0} \in \GC$ arbitrary. Take $n(m)$ as in the condition.
			Then $\ell_m \in \pi_{m,n(m)}(\gamma_n)$, so the (two-sided) $f$-orbit of
			$(\gamma_n)_{n \geq 1}$ intersects $U$, proving minimality.
			(To show that the forward orbit of $(\gamma)_{n\geq 0}$ also visits
			$U$, we can take the successor loop $\gamma'_{n(m)}$ under $f$ and use the same argument.)
			
			To prove the ``only if''-direction, assume that for every $m$ and $n > m$, there are
			$\ell_m \in \GC_m$ and $\ell_n \in \GC_n$ such that $\pi_{m,n}(\ell_n) \not\owns \ell_m$.
			Since $\GC_m$ is finite, we can find a fixed $\ell \in \GC_m$ and a subsequence $n_k$
			such that $\pi_{m,n_k}(\ell_{n_k}) \not\owns \ell$.
			In particular $\ell \notin \pi_{m,n}(\gamma_n)$ for every
			$m < n \leq n_k$ and $\gamma_n \in \pi_{n,n_k}(\ell_{n_k})$.
			Take any accumulation point of  $(\gamma_n)_{n \geq 1}$
			of $\bigcap_{r \geq 1} \bigcup_{k \geq r} [\ell_{n_k}]$,
			where $ [\ell_{n_k}]$ indicates the cylinder set, then
			the $f$-orbit of $(\gamma_n)_{n \geq 1}$ is disjoint from
			$[\ell]$, so minimality fails.
			
			The property of minimality allows us to telescope the graph cover such that the winding matrices become strictly positive.
			
			\item Unique ergodicity is covered in \cite[Corollaries 4.24 and 4.25]{Shi16}, but the proof in terms of the winding matrices is no
			different from those for Bratteli-Vershik systems, see \cite{BKK17,BKMS13} and \cite[Section 6.3.3]{B23}.
		\end{enumerate}
	\end{proof}

	\textbf{Questions:}
	Does there exist a characterization of shadowing in terms of graph covers, cf.\ \cite{GM06}?
	How can one derive the spectrum of the Koopman operator $U_f$ from the shape of the graph cover?
	How is the dimension group computed?
	
	\subsection{Uniform rigidity}
	
	Rigidity, as a measure-theoretic concept, was introduced by Furstenberg \& Weiss \cite{FW73}, but its topological version (uniform rigidity) seems to have developed first in a paper by Glasner \& Maon \cite{GM89} and used also in \cite{GW93}. 
	
	\begin{definition}
		A map $f$ on a compact metric space $(X,d)$ is called
		{\bf uniformly rigid} if for every $\eps > 0$ there is an iterate $k \geq 1$ such that
		$d(x, f^k(x)) < \eps$ for all $x \in X$.
	\end{definition}

	\begin{theorem}\label{thm:ur}
		Let $(X,g)$ be a transitive continuous map on a Cantor set.
		Let $(\GC,f)$  be a graph cover representation of $(X,g)$.
		Then $g$ is uniformly rigid if and only if for every $n \in \N$,
		there is a loop $C_n = (C_n^i)_{i=0}^{\# C-1}$ (i.e., $\sbv(C_n^0) = \tbv(C_n^{\#C-1})$ but self-intersections are allowed) covering every edge of $\GC_n$,
		such that for every $y \in \GC$ there is $i \in \N$
		such that $f^j(y) \in [\sbv(C_{n}^{i+j \pmod {\#C}})]$ for all $j \geq 0$.
	\end{theorem}
	
	It follows that $\pi_{n+1}:C_{n+1} \to C_n$ is a $q_n$-cover of loops or some $q_n \in \N$,
	so $(\GC,f)$ is a factor of the $(q_n)$-odometer.
	Factors of odometers are (possibly degenerate) odometers themselves,
	so each transitive uniformly rigid map on a Cantor set is a
	(possibly degenerate) odometer. This makes the condition in Theorem~\ref{thm:ur}
	also a criterion for $(\GC,f)$ to represent an odometer, and also shows that transitive uniformly recurrent maps on a Cantor set are odometers.
	\\[4mm]
	\begin{proof}
		$\Leftarrow$: Choose $\eps > 0$ arbitrary and $n$ such that $2^{-n} < \eps$; hence all cylinder sets $[\gamma]$, $\gamma \in V_n$,
		have diameter $< \eps$.
		Take $k = \# C_n$. Then for every $x \in \GC$, $f^k(x)$ and $x$ belong to the same cylinder $[\gamma]$, $\gamma \in V_n$.
		Hence $d(f^k(x), x) < \eps$.
		
		$\Rightarrow$: Choose $\eps > 0$ arbitrary and $n \in \N$ such that $x,y \in \GC$ are $\eps$-close if and only if they belong to the same cylinder $[\gamma]$, $\gamma \in V_{n}$.
		Let $k \geq 1$ be such that $d(f^k(x),x) < \eps$ for all $x \in \GC$.
		
		Take $x \in \GC$ with a dense orbit.
		Let $C_n^i = (f^i(x)_n,f^{i+1}(x)_n)$ for $0 \leq i < k$
		and $C_n = \{ C_n^i \}_{i=0}^{k-1}$.
		Since $f^i(x) \in [\sbv(C_n^i)]$ and $f^{i+mk}(x)$ and $x$ are in the same cylinder $[C_n^i]$ for all $m \geq 0$, we have
		$\orb_f(x) \subset \bigcup_{i=0}^{k-1} [\sbv(C_n^i)]$.
		But $\orb_f(x)$ is dense in $\GC$, so $\{ C_n^i\}_{i=0}^{k-1} \supset E_n$.
		
		Now take $y \in \GC$ and $j \geq 0$ arbitrary, and find $i \in \{0, \dots, k-1\}$ and sequence $(r_t)_{t \geq 1}$ such that $f^{i+r_tk}(x) \to y$, and hence $y \in [\sbv(C_n^i)]$. For $t$ sufficiently large, also
		$d(f^{i+r_t k+j}(x), f^j(y)) < \eps$, so $y \in [\sbv(C_n^{i+j \pmod k})]$.
	\end{proof}

	\textbf{Questions:} How to characterize graph covers that are rigid from a measure-theoretical point of view:
	for every $\eps > 0$ and $A \in \cB$ with $\mu(A) > 0$ there is $n$ such that
	$\mu(A \triangle T^{-n}(A)) < \eps$.
	
	\subsection{Linear recurrence}
	
	To any graph cover with $\GC_0$ consisting of loops $\{a_1, \dots, a_d\}$ of weight $1$, we can assign a subshift $(X,\sigma)$
	by taking the itinerary $\iota:\bg = (\gamma_{n})_{n \geq 0} \mapsto (x_k)_{k \geq 0}$
	where $x_k = j$ if $f^k( \bg )_0 = a_j$.
	Then $X = \iota(\GC)$ and $\sigma \circ \iota = \iota \circ f$.
	
	\begin{definition}
		A subshift $(X,\sigma)$ is {\bf linearly recurrent} if there is $L > 0$ such that for every $x \in X$ every subword $u$ appearing in $x$ reappears in $x$ with gap at most $L|u|$.
	\end{definition}
	
	Linearly recurrent subshifts have zero entropy (in fact, they have linear word-complexity) and, if transitive, they are minimal, see \cite{DHS99}. As the gaps between word $u$ are finite,
	the following notion becomes useful:
	
	\begin{definition}\label{def:returnword}
		Given a word $u$, we call $w$ a {\bf return word} for $u$ if
		\begin{itemize}
			\item $u$ is a prefix and suffix of $wu$ but $u$ does not occur elsewhere in $wu$;
			\item $wu \in \cL(X)$, the language of the subshift $X$.
		\end{itemize}
		We denote the collection of return words of $u$ by $\cR_u$.
	\end{definition}
	
	In fact, $\# \cR_u \leq L(L+1)^2$, see \cite[Theorem 24]{DHS99}.
	The next theorem gives a sufficient condition for linear recurrence in terms of graph covers, and is inspired by \cite{Dur00}.
	We will assume that the graph cover $(\GC,f)$ is weighted and minimal, and indeed telescoped in such a way that $\pi_{n+1}(\gamma)$ covers all the edges of $\GC_n$ for all $n \geq 0$ and $\gamma \in \GC_{n+1}$.
	
	\begin{theorem}
		Given\footnote{For minimal systems, the gap-size is in fact independent of the choice of $\bg$.}  $\bg \in \GC$, define its gap-size
		$g^{(m)}(j) = \min\{i \geq 1 : f^j(\bg)_mf^{j+1}(\bg)_m
		= f^{j+i}(\bg)_mf^{j+i+1}(\bg)_m\}$.
		If there are only finitely many different winding matrices and
		$$
		D := sup\{g^{(m)}(j) : j \geq 1,m \geq 0\} < \infty,
		$$
		then the graph cover $(\GC,f)$ represents a  linearly recurrent subshift.
	\end{theorem}
	
	\begin{proof}
		First, recall that $\pi_{n+1}(\gamma) \supset \GC_n$ for all $\gamma \in \GC_{n+1}$. Let $W^n$, $n \geq 1$ denote the winding matrices. Define
		$$
		K_1 := \max\left\{  \sum_{\gamma \in \GC_n} W^n_{\gamma, \gamma'} : n \geq 1, \gamma' \in \GC_{n-1}\right\}
		$$
		and
		$$
		K_2 := \min\left\{ \sum_{\gamma \in \GC_n} W^n_{\gamma, \gamma'} : n \geq 1, \gamma' \in \GC_{n-1}\right\} > 0.
		$$
		Write $K = K_2/K_1$. Then, for all $n \geq 1$ and $a,b \in \GC_n$:
		\begin{eqnarray*}
			\frac{w(a)}{w(b)}
			&=&
			\frac{(\1 W^n \cdots W^1)_a}{(\1 W^n \cdots W^1)_b}
			\leq
			\frac{K_2 \min_{c \in \GC_n} (\1 W^{n-1} \cdots W^1)_c}{K_1 \min_{c \in \GC_n} (\1 W^{n-1} \cdots W^1)_c} \leq K.
		\end{eqnarray*}
		Let $u$ be a subword of  $X$ such that $|u| \geq \min_{a \in \GC_N} w(a)$
		and $N'> N$ be minimal such that $|u| \leq \min_{a \in \GC_{N'}} w(a)$.
		In particular,
		$|u| > \min_{a \in  \GC_{N'-1}} w(a)$ and
		every appearance of $u$ is inside some word
		$\iota ( \pi_1 \circ \cdots \circ \pi_{N'}(ab))$, $ab \in \GC_{N'}^2$.
		Let $v$ be a return word to $u$.
		Since each path $ab$ in $\GC_{N'}$ appears with a gap at most $ D$,
		we have
		\begin{eqnarray*}
			|v| &\leq& D \max_{ c \in \GC_{N'}} w(c)
			\leq D\, K \min_{ c \in \GC_{N'}}  w(c)  \\
			&\leq& D\, K \max_{ c \in \GC_{N'-1}}  w(c)
			\cdot  \min_{ c \in \GC_{N'-1}} \sum_{b \in \GC_{N'}} W^{N'}_{b,c} \\
			&\leq& D\, K^2 \min_{ c \in \GC_{N'-1}} w(c)
			\cdot \min_{ c \in \GC_{N'-1}} \sum_{b \in \GC_{N'}} W^{N'}_{b,c}
			\\
			&\leq& D\, K^2 \ |u| \ \min_{ c \in \GC_{N'-1}} \sum_{b \in \GC_{N'}} W^{N'}_{b,c}.
		\end{eqnarray*}
		Since the lengths of the return words give the gaps between appearances of $u$,  linear recurrence follows with constant
		$$
		L = DK^2\max_{n \geq N} \min_{ c \in \GC_{n-1}} \sum_{b \in \GC_n} W^{n}_{b,c}.
		$$
	\end{proof}

	\section{Rauzy Graphs}
	
	Rauzy graphs were initiated by G\'erard Rauzy \cite{Rauzy}, and later frequently applied (e.g.\ \cite{Frid, Rote}) to construct and study subshifts with particular (complexity) properties.
	Every subshift $(X,\sigma)$ on the (finite) alphabet $\Alf$ can be represented by a sequence of {\bf Rauzy graphs}, which we will denote by $(\GC_n)_{n \geq 0}$ again.
	Here $\GC_0$ consist of a single vertex labeled $\epsilon$ (for the empty word) and $\#\Alf$ loops.
	For $n \geq 1$, the vertex set of $\GC_n$ is $\cL_n(X)$, i.e.,
	set of the words in $X$
	of length $n$, and the arrow set is $\cL_{n+1}(X)$, i.e.,
	the set of words of length $n+1$.
	If $\gamma \in \cL_{n+1}$ we write $\ld{\gamma}$ and $\rd{\gamma}$
	for the word $\gamma$ with the first and last letter removed, respectively.
	The source and target of the arrow $\gamma$ are
	$\sg(\gamma) = \rd{\gamma}$ and $\tg(\gamma) = \ld{\gamma}$
	or in other words
	$\gamma = (\rd{\gamma} \to \ld{\gamma})$.
	
	It is straightforward to turn this sequence of Rauzy graphs into  a graph cover, namely by the bonding maps $\pi_n$ given by
	$$
	\pi_n(\gamma) 
	= \rd{\gamma}.
	$$
	On the level of vertices, this works by ``inclusion'':
	$\pi_n(u) = \rd{u}$.
	These bonding maps are automatically edge surjective and positive directional, but in general not negative directional.

	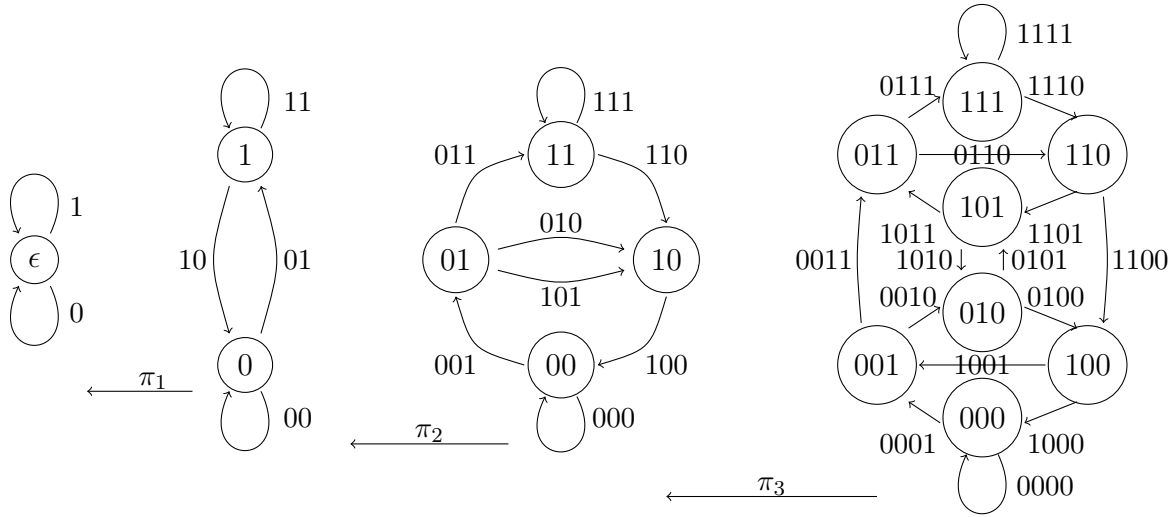
\begin{figure}[ht]
		\begin{center}
			\begin{tikzpicture}[scale=0.7]
			\node[circle, draw=black] at (-1,3) {$\epsilon$};
			\draw[->, draw=black] (-0.7, 3.5) .. controls (0, 5) and (-2, 5) .. (-1.3, 3.5);
			\node at (-0.2,4) {\small $1$};
			\draw[->, draw=black] (-0.7, 2.5) .. controls (0, 1) and (-2, 1) .. (-1.3, 2.5);
			\node at (-0.2,2) {\small $0$};
			\draw[->, draw=black] (2, 0.5) -- (0, 0.5); \node at (1.25,0.7) {\small $\pi_1$};
			\node[circle, draw=black] at (3,5) {$1$};
			\draw[->, draw=black] (3.3, 5.5) .. controls (4, 7) and (2, 7) .. (2.7, 5.5);
			\node at (4,6) {\small $11$};
			\draw[->, draw=black] (3.3, 1.6) .. controls (3.7, 3) .. (3.3, 4.4); \node at (4,3) {\small $01$};
			\draw[->, draw=black] (2.7, 4.4) .. controls (2.3, 3) .. (2.7, 1.6); \node at (2,3) {\small $10$};
			\node[circle, draw=black] at (3,1) {$0$};
			\draw[->, draw=black] (3.3, 0.5) .. controls (4, -1) and (2, -1) .. (2.7, 0.5);
			\node at (4,0) {\small $00$};
			\draw[->, draw=black] (8, -0.5) -- (5, -0.5); \node at (6.5,-0.3) {\small $\pi_2$};
			\node[circle, draw=black] at (9,5) {$11$};
			\node[circle, draw=black] at (7,3) {$01$};
			\node[circle, draw=black] at (11,3) {$10$};
			\draw[->, draw=black] (9.3, 5.6) .. controls (10, 7) and (8, 7) .. (8.7, 5.6);
			\node at (10,6) {\small $111$};
			\draw[->, draw=black] (7.8, 3.2) .. controls (9, 3.5) .. (10.2, 3.2); \node at (9,3.7) {\small $010$};
			\draw[->, draw=black] (7.8, 2.8) .. controls (9, 2.5) .. (10.2, 2.8); \node at (9,2.3) {\small $101$};
			\node[circle, draw=black] at (9,1) {$00$};
			\draw[->, draw=black] (9.3, 0.4) .. controls (10, -1) and (8, -1) .. (8.7, 0.4);
			\node at (10,0) {\small $000$};
			\draw[->, draw=black] (7, 3.7) .. controls (7.3, 4.7) .. (8.3,5); \node at (7,5) {\small $011$};
			\draw[->, draw=black] (8.3,1) .. controls (7.3, 1.3) .. (7, 2.3); \node at (7,1) {\small $001$};
			\draw[->, draw=black] (9.7, 5) .. controls (10.7, 4.7) .. (11,3.7); \node at (11,5) {\small $110$};
			\draw[->, draw=black] (11, 2.3) .. controls (10.7, 1.3) .. (9.7, 1); \node at (11,1) {\small $100$};
			\draw[->, draw=black] (15, -1.5) -- (11, -1.5); \node at (13,-1.3) {\small $\pi_3$};
			\node[circle, draw=black] at (15,5) {$011$};
			\draw[->, draw=black] (17.3, 6.8) .. controls (18, 8.2) and (16, 8.2) .. (16.7, 6.8);
			\node at (18.2,7.3) {\small $1111$};
			\draw[<-, draw=black] (19.3, 1.8) .. controls (19.4, 3) .. (19.3, 4.2); \node at (20,3) {\small $1100$};
			\draw[<-, draw=black] (14.7, 4.2) .. controls (14.6, 3) .. (14.7, 1.8); \node at (14,3) {\small $0011$};
			\node[circle, draw=black] at (15,1) {$001$};
			\draw[->, draw=black] (17.3, -0.7) .. controls (18, -2.2) and (16, -2.2) .. (16.7, -0.7);
			\node at (18.2,-1.3) {\small $0000$};
			\node at (17,5) {\small $0110$};\node at (17,1) {\small $1001$};
			\node[circle, draw=black] at (17,4) {$101$};
			\node[circle, draw=black] at (17,2) {$010$};
			\node[circle, draw=black] at (17,6) {$111$};
			\node[circle, draw=black] at (17,0) {$000$};
			\node[circle, draw=black] at (19,5) {$110$};
			\node[circle, draw=black] at (19,1) {$100$};
			\draw[->, draw=black] (17.4, 2.8) -- (17.4, 3.2); \node at (18.1,3) {\small $0101$};
			\draw[<-, draw=black] (16.6, 2.8) -- (16.6, 3.2); \node at (15.9,3) {\small $1010$};
			
			\draw[->, draw=black] (15.8, 5) -- (18.2, 5);
			\draw[->, draw=black] (18.2, 1) -- (15.8, 1);
			\draw[->, draw=black] (15.6, 5.7) -- (16.2, 6.1); \node at (15.6,6.3) {\small $0111$};
			\draw[->, draw=black] (17.8, 6.1) -- (18.8, 5.7); \node at (18.4,6.3) {\small $1110$};
			\draw[->, draw=black] (16.2, -0.1) -- (15.6, 0.3) ; \node at (18.4,-0.5) {\small $1000$};
			\draw[->, draw=black] (18.8, 0.3) -- (17.8, -0.1); \node at (15.6,-0.5) {\small $0001$};
			\draw[->, draw=black] (16.2, 3.9) -- (15.6, 4.3) ; \node at (18.4,3.5) {\small $1101$};
			\draw[->, draw=black] (18.8, 4.3) -- (17.8, 3.9); \node at (15.6,3.5) {\small $1011$};
			\draw[->, draw=black] (15.6, 1.7) -- (16.2, 2.1); \node at (15.6,2.3) {\small $0010$};
			\draw[->, draw=black] (17.8, 2.1) -- (18.8, 1.7); \node at (18.4,2.3) {\small $0100$};
			\end{tikzpicture}
			\caption{Graph cover for the one-sided full shift on $\Alf = \{ 0,1\}$.}
			\label{fig:gc_ful_shift}
		\end{center}
	\end{figure}

	\subsection{On the number of ergodic $f$-invariant measures}
	
	In analogy to shifts, we call a vertex of a graph
	{\em left-special} if it has at least two incoming edges, {\em right-special} if it has at least two outgoing edges, and
	{\em bi-special} if it is both left-special and right-special.
	
	\begin{prop}\label{prop:gap}
		Let $\GC_n$ be the Rauzy graph of a subshift over the alphabet
		$\{ 0,1\}$, with $f$-invariant measure $\mu$. 
		Assume that there are $l_n$ left-special vertices, $r_n$ right-special vertices and $b_n$ bi-special vertices.
		Then the number $M_n$ of different values $\mu(Z)$ for $n$-cylinders
		is at most $3(r_n+b_n)$.
	\end{prop}
	
	\begin{proof}
		First note that $l_n = r_n$, since for the two-letter alphabet, each vertex has at most two incoming and two outgoing arrows.
		We can partition $\GC_n$ in paths or vertices that are
		\begin{itemize}
			\item paths starting with a left-special vertex and ending in a right-special vertex (both included);
			\item paths starting with a left-special vertex (included) to a left-special vertex (not included);
			\item paths from a right-special vertex (not inclusive) ending with a right-special vertex (included);
			\item bi-special vertices.
			(There is no change in the result if we pull every bi-special vertex into a left-special and right-special vertex connected by a single arrow.)
		\end{itemize}
		By invariance of the measure, all the vertices (recall that the vertices of $\GC_n$ for the $n$-cylinder sets) in each of the above paths must have the same mass.
		Simple counting shows that the number of such paths at most $3(r_n + b_n)$, and the proposition follows.
	\end{proof}
	
	Special cases are Sturmian shifts, where this result is known as the three-gap theorem, originally proven by S\'os, \cite{Sos57,Sos58}.
	For the full shift, every vertex is bi-special, and
	$M_n = p_n$ is sharp for an $n$-step memory Markov measure.

	\begin{prop}\label{prop:erg_meas}
		Let $\GC_n$ be the Rauzy graph of a subshift over the alphabet
		$\{ 0,1\}$.
		Assume that there are $l_n$ left-special vertices, $r_n$ right-special vertices and $b_n$ bi-special vertices.
		Then the number of ergodic invariant probability  measures
		is bounded by $N := \liminf_{n\to\infty} r_n+b_n + 1$.
	\end{prop}
	
	This bound is comparable with the bound $\liminf_n \#\{ \text{simple loops in } \GC_n\}$ from \cite{GM06},
	but sometimes an over-estimate. For Sturmian shifts, it would give at most two ergodic measures, although  Sturmian shifts are uniquely ergodic.
	However, if each $\GC_n$ is a vertex with two loops (so $b_n = 1$) and
	the winding matrices are 
	$$
	 W^n = \begin{pmatrix}n^2 & 1 \\ 1 & n^2\end{pmatrix},
	$$ 
	then the corresponding graph cover indeed has two measures,
	see \cite{BKK17}.
	
	In \cite[Corollary 4.3]{DF17} the sharper bound is given that translates to $N = \liminf_{n\to\infty} r_n+b_n$, but their result applies only to subshifts, and that puts a uniform bound on the entries of the winding matrices.
	\\[4mm]
	\begin{proof}
		As in the proof of Proposition~\ref{prop:gap},
		we partition the graphs $\GC_n$ in paths $p$. We can replace
		each bi-special vertex by a left-special vertex and right-special vertex connected with an arrow.
		Write $\sg(p)$ for the special vertex at the beginning and $\tg(p)$ for the special vertex at the end of $p$.
		If $\mu$ is an $f$-invariant measure on $\GC_n$, then the following equations must hold:
		$$
		\begin{cases}
		\mu(p) + \mu(p') = \mu(p'') & \text{ if } \sg(p) = \sg(p') = \tg(p''),\\
		\mu(p) = \mu(p') + \mu(p'') & \text{ if } \tg(p) = \sg(p') = \sg(p'').
		\end{cases}
		$$
		Writing this as matrix equation, we find $\mu = M \mu$, where we consider $\mu$ as a column vector with $\mu(p)$ as entries.
		The matrix $M = (m_{pp'})$ has the following relations:
		$$
		\begin{cases}
		m_{pp} = m_{pp'} = -1, \ m_{pp''} = 1 & \text{ if } \sg(p) = \sg(p') = \tg(p''),\\
		m_{pp} = -1, \  m_{pp'} = m_{pp''} = 1 & \text{ if } \tg(p) = \sg(p') = \sg(p'').
		\end{cases}
		$$
		Gaussian elimination in the matrix $M$ gives that the null-space has dimension at most $r_n + 1$, and contains a unit simplex of dimension at most $r_n$.
		Ergodic measures correspond to the extremal points of these simplices, of which there are at most $r_n+1$.
		
		Let $(n_k)_{k \in \N}$ be the subsequence such that $r_{n_k}+b_{n_k}+1 = N$.
		Each ergodic measure $\mu$ of the graph cover restricts to a
		measure $\mu_{n_k}$ on $\GC_0 \leftarrow \GC_0 \leftarrow \dots \leftarrow \GC_{n_k}$. Let $\cF_n$ denote the algebra generated by the  $n$-cylinders. Hence, the atoms of
		$\mu_{n_k}$ are exactly the atoms of $\cF_{n_k}$, and
		the conditional expectation $\E(\mu_{n_{k+1}} | \cF_{n_k}) = \mu_{n_k}$.
		Passing to a subsequence if necessary, the Martingale Theorem
		gives that $\lim_k \mu_{n_k} = \mu$, and in fact $\mu$ is ergodic.
		But for each $k$, there are only $N$ possibilities for $\mu_{n_k}$,
		so there cannot be more than $N$ ergodic measures.
	\end{proof}

	\subsection{Rauzy graphs of the Sturmian shift}
	
	The {\bf Rauzy graph}\index{Rauzy graph} $\GC_n$ of a Sturmian subshift\index{Sturmian shift} $X$ is the word-graph
	in which the vertices are the words $u \in \cL_n(X)$ and there is
	an arrow $u \to u'$ if $ua = bu'$ for some $a,b \in \{ 0 , 1\}$.
	Hence $\GC_n$ has $p(n)=n+1$ vertices and $p(n+1)= n+2$ edges; it is the vertex-labeled
	transition graph of the $n$-block shift\index{block shift}\index{shift!block} interpretation of the Sturmian shift.
	
	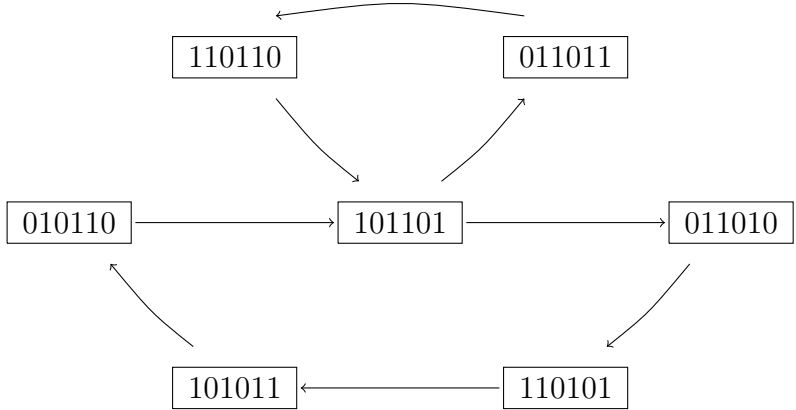
\begin{figure}[ht]
		\begin{center}
			\begin{tikzpicture}[scale=1.1]
			\draw[-] (4.25,1.75) -- (5.75,1.75) -- (5.75,2.25) -- (4.25,2.25) -- (4.25,1.75); \node at (5,2){$101101$};
			\draw[->, draw=black] (5.5, 2.5) .. controls (6, 2.9) .. (6.5, 3.5);
			\draw[-] (6.25, 3.75) -- (7.75, 3.75) -- (7.75, 4.25) -- (6.25, 4.25) -- (6.25, 3.75); \node at (7,4){$011011$};
			\draw[->, draw=black] (6.5, 4.5) .. controls (5, 4.7) .. (3.5, 4.5);
			\draw[-] (2.25, 3.75) -- (3.75, 3.75) -- (3.75, 4.25) -- (2.25, 4.25) -- (2.25, 3.75); \node at (3,4){$110110$};
			\draw[->, draw=black] (3.5, 3.5) .. controls (4, 2.9) .. (4.5, 2.5);
			\draw[->, draw=black] (5.8, 2) .. controls (6.5, 2) .. (8.2, 2);
			\draw[-] (8.25, 1.75) -- (9.75, 1.75) -- (9.75, 2.25) -- (8.25, 2.25) -- (8.25, 1.75); \node at (9,2){$011010$};
			\draw[->, draw=black] (8.5, 1.5) .. controls (8, 0.9) .. (7.5, 0.5);
			\draw[-] (6.25, -0.25) -- (7.75, -0.25) -- (7.75, 0.25) -- (6.25, 0.25) -- (6.25, -0.25); \node at (7,0){$110101$};
			\draw[->, draw=black] (6.2, 0) .. controls (5, 0) .. (3.8,0);
			\draw[-] (2.25, -0.25) -- (3.75, -0.25) -- (3.75, 0.25) -- (2.25, 0.25) -- (2.25, -0.25); \node at (3,0){$101011$};
			\draw[->, draw=black] (2.5, 0.5) .. controls (2, 0.9) .. (1.5, 1.5);
			\draw[-] (0.25, 1.75) -- (1.75, 1.75) -- (1.75, 2.25) -- (0.25, 2.25) -- (0.25, 1.75); \node at (1,2){$010110$};
			\draw[->, draw=black] (1.8, 2) .. controls (2.5, 2) .. (4.2, 2);
			\end{tikzpicture}
			\caption{The Rauzy graph $\GC_6$ based on the Fibonacci Sturmian sequence
			}
			\label{fig:Rauzy}
		\end{center}
	\end{figure}
	
	In the example of Figure~\ref{fig:Rauzy}, the word
	$u = 101101$ is bi-special, but only
	$0u0, 0u1$ and $1u0 \in \cL(X)$ (i.e., $u$ is a regular bi-special word).\index{regular bi-special word}
	Since $p(n+1)-p(n) = 1$ for a Sturmian sequence, every Rauzy graph contains exactly
	one left-special and one right-special word, and they may be merged into a
	single bi-special word.
	Hence, there are two types of Rauzy graphs, see Figure~\ref{fig:Rauzy_types}.
	
	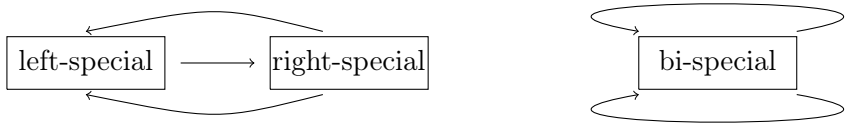
\begin{figure}[ht]
		\begin{center}
			\begin{tikzpicture}[scale=0.7]
			\draw[->] (16, 2.1) .. controls (20, 2.8) and (9,2.8) .. (13, 2.1);
 			\draw[->] (16, 0.9) .. controls (20, 0.2) and (9,0.2) .. (13, 0.9);
			\draw[-] (1,1) -- (4,1) -- (4,2) -- (1,2) -- (1,1); \node at (2.5,1.5){\small left-special};
			\draw[-] (6,1) -- (9,1) -- (9,2) -- (6,2) -- (6,1); \node at (7.5,1.5){\small right-special};
			\draw[->] (4.3,1.5) -- (5.7,1.5);
 			\draw[->] (7, 2.1) .. controls (5, 2.6) .. (2.5, 2.1);
 			\draw[->] (7, 0.9) .. controls (5, 0.4) .. (2.5, 0.9);
 			\draw[-] (13,1) -- (16,1) -- (16,2) -- (13,2) -- (13,1); \node at (14.5,1.5){\small bi-special};
			\end{tikzpicture}
			\caption{The two types of Rauzy graphs for a Sturmian sequence.}
			\label{fig:Rauzy_types}
		\end{center}
	\end{figure}
	
	The transformation from $\GC_n$ to $\GC_{n+1}$ is as follows:
	\begin{itemize}
		\item[(a)] If $\GC_n$ is of the first type, then the middle path decreases by one vertex,
		and the upper and lower path increase by one vertex. The left-special vertex of $\GC_n$
		splits into two vertices,
		with outgoing arrows leading to the previous successor vertex which now becomes left-special.
		Similarly, the right-special vertex of $\GC_n$ is split into two vertices with incoming arrows
		emerging from the previous predecessor vertex, which now becomes right-special.
		\item[(b)] If $\GC_n$ is of the second type, then one of the two paths becomes
		the central path in $\GC_{n+1}$, the other path becomes the upper path
		of $\GC_{n+1}$, and there is an extra arrow in $\GC_{n+1}$ from
		the right-special word to the left-special word. Thus the bi-special vertex of $\GC_n$ is split into two vertices,
		one of which becomes left-special in $\GC_{n+1}$, and one of the predecessors of the bi-special vertex in $\GC_n$
		becomes right-special in $\GC_{n+1}$.
	\end{itemize}

	We can combine all Rauzy graphs into a single inverse limit space
	$$
	\GC = \IL(\GC_n,\pi_n)
	= \{ (\gamma_n)_{n \geq 0} : \pi_{n+1}(\gamma_{n+1}) = \gamma_n \in \GC_n \text{ for all } n \geq 0\},
	$$
	where $\GC_0$ has only on vertex $\epsilon$ and one arrow $\epsilon \to \epsilon$,
	and $\pi_{n+1}:\GC_{n+1} \to \GC_n$ is the prefix map $\pi_{n+1}(ua) = u$ for every
	$u \in \cL_n(X)$, $a \in \Alf$.
	It is the inverse of the map described in items (a) and (b) above.

	\section{Kakutani-Rokhlin partitions}

	An important tool in the study of dynamical systems $(X,T)$ are
	(sometimes called induced maps) $T_Y:Y \to Y$ to subsets $Y \subset X$.
	These are defined by
	$$
	T_Y(y) = T^r(y)\ \text{ for the {\bf return time} }\ r = r(y) := \min\{ i > 0 : T^i(y) \in Y\}.
	$$
	The 
	exhaustion of the space (if $(X,T)$ is minimal) is
	\begin{equation}\label{eq:exhaust}
	X = \sqcup_{i \geq 0} T^i(\{ y \in Y : i < r(y) \}),
	\end{equation}
	and the Rokhlin Lemma \cite[Theorem 3.10]{Rudolph} in the measure-preserving setting
	are classical techniques associated with first return maps.
	
	For continuous minimal (or at least aperiodic) transformations of a Cantor set, this led to
	a generalization of \eqref{eq:exhaust}, called Kakutani-Rokhlin partition.
	The seminal paper is by Herman et al.\ \cite{HPS92}, who coined the name.
	
	\begin{definition}\label{def:KR-partition}
		Let $(X,T)$ be a continuous dynamical system on a Cantor set. A
		{\bf Kakutani-Rokhlin (KR) partition}\index{Kakutani-Rokhlin partition|textbf}\index{KR-partition (Kakutani-Rokhlin)} is a partition
		$$
		\cP = \Big\{ T^j(B_i) \Big\}_{i=1, j=0}^{N, h_i - 1}
		$$
		of $X$ into clopen\index{clopen} sets that are pairwise disjoint and together cover $X$.
		We call $B = \cup_{i=1}^N B_i$ the {\bf base}\index{base} of the KR-partition
		and the integers $h_i$ the {\bf heights}\index{height}.
		Also we assume that $T^{h_i}(B_i) \subset B$. (If $T$ is invertible, this is automatic.)
	\end{definition}
	Usually we need a sequence $(\cP_n)_{n \geq 0}$ of KR-partitions, with bases $B(n)$
	and height vectors $(h_i(n))_{i=1}^{N_n}$, having the following properties:
	\begin{enumerate}
		\item[{\bf (KR1)}] The sequence of bases is nested: $B(n+1) \subset B(n)$, and
		\item[{\bf (KR2)}] $\cP_{n+1} \succeq \cP_n$, that is: $\cP_{n+1}$ refines $\cP_n$.
		\item[{\bf (KR3)}] $\bigcap_n B(n)$ is a single point.
		\item[{\bf (KR4)}] $\{ A \in \cP_n : n \in \N\}$ is a basis of the topology\index{basis of the topology}.
	\end{enumerate}
	The following property (KR5) relies crucially on minimality. Property (KR6) is optional, but ensures that
	there is a unique smallest path in the context of Bratteli-Vershik systems\index{Bratteli-Vershik system}
	and cutting-and-stacking systems\index{cutting-and-stacking}.
	\begin{enumerate}
		\item[{\bf (KR5)}] For all $n \in \N$, $i \leq N(n)$ and $i' \leq N(n-1)$, there is $0 \leq j < h_i(n)$
		such that $T^j(B_i(n)) \subset B_{i'}(n-1)$.
		\item[{\bf (KR6)}] $B(n)\subset B_1(n-1)$ for all $n \in \N$.
	\end{enumerate}
	
	The transition from nested sequence of Kakutani-Rokhlin partitions to graph covering
	is fairly direct: The vertices of the $n$-th graph $\GC_n$ are the elements of $\cP_n$,
	and there are arrows $T^j(B_i(n)) \to T^{j+1}(B_i(n))$ if $0 \leq j < h_i(n)-1$,
	and also $T^{h_i(n)-1} (B_i(n)) \to B_{i'}(n)$ if $T^{h_i(n)} (B_i(n)) \cap B_{i'}(n) \neq \void$.
	As bonding maps $\pi_n:\GC_n \to \GC_{n-1}$ we take the inclusion:
	$\pi_n(A) = A'$ for vertices $A$ of $\GC_n$
	and $A'$ of $\GC_{n-1}$ if $A \subset A'$.
	Then inverse limit space
	$$
	\GC = \IL(\GC_n,\pi_n)
	= \{ (\gamma_n)_{n \geq 0} : \pi_{n+1}(\gamma_{n+1}) = \gamma_n \in \GC_n \text{ for all } n \geq 0\},
	$$
	is the graph covering and (KR6) provides the positive directional property\index{positive directional}.
	The dynamics $f:\GC \to \GC$ is by following the arrows as in equation~\eqref{eq:follow-arrow}.
	
	\section{Bratteli-Vershik systems}

	An {\bf ordered Bratteli diagram}\index{ordered Bratteli system} 
	$(E_i,V_i,<)_{i \geq 1}$ is an infinite graph consisting
	of
	\begin{itemize}
		\item a sequence of finite non-empty vertex sets $V_i$, $i \geq 0$, and there is an additional $V_0$ consisting of a single vertex $v_0$;
		\item a sequence finite non-empty edge sets $E_i$, $i \geq 1$, such that
		each edge $e \in E_i$ connects a vertex $\sbv(e) \in V_{i-1}$ to a vertex $\tbv(e)
		\in V_i$. (Here $\sbv$ and $\tbv$ stand for {\bf source} and {\bf target}.)
		For every $v \in V_{i-1}$, there exists at least one outgoing edge $e \in E_i$
		with $v = \sbv(e)$, and for every $v \in V_i$ there exists at least one
		incoming edge $e \in E_i$ with $v = \tbv(e)$;
		\item for each $v \in \cup_{i \geq 1} V_i$,
		a total order $<$ between its incoming edges.
	\end{itemize}
	The path space
	$$
	X_{\BV} := \{ (x_i)_{i \geq 1} : x_i \in E_i, \tbv(x_i) = \sbv(x_{i+1}) \text{ for all } i \in \N\}
	$$
	is the collection of all infinite edge-labeled\index{edge-labeled}
	paths starting from $v_0$, endowed with product topology\index{product topology}.
	That is, the set of infinite paths with a common initial $n$-path is clopen\index{clopen}, and all sets of this type form a basis of the topology.

	The {\bf Vershik adic transformation} ({\bf Vershik map})\index{Vershik map} $\tau:X_{\BV} \to X_{\BV}$
	is defined as follows \cite{Ver81}:
	For $x \in X_{\BV}$, let $i$ be minimal such that $x_i \in E_i$
	is not the maximal incoming edge.
	Then put
	$$
	\begin{cases}
	\tau(x)_j = x_j \quad \text{ for  } j > i, \\
	\tau(x)_i \text{ is the successor of } x_i \text{ among all incoming edges at this level},\\
	\tau(x)_1 \dots \tau(x)_{i-1} \text{ is the minimal path
		connecting } v_0 \text{ with } \sbv(\tau(x)_i).
	\end{cases}
	$$
	If no such $i$ exists, then $x \in X_{\BV}^{\max}$, and we need to choose
	$y \in X_{\BV}^{\min}$ to define $\tau(x) = y$.
	Whether $\tau$ extends continuously
	to $X_{\BV}^{\max}$ depends on how well we can make this choice.
	Medynets \cite{Med06} gave an example of a Bratteli diagram that doesn't allow any ordering
	by which $\tau$ is continuously extendable, even if $\# X_{\BV}^{\min} = \# X_{\BV}^{\max}$.
	For this diagram the only incoming edges to $u \in V_n$ come from $u \in V_{n-1}$, and therefore there is a minimal and a maximal path going through
	vertices $u$ only. By the same token, there is a minimal and a maximal path going through vertices $w$ only.
	No matter how $\tau$ is defined on these two maximal paths, there is no way of putting an order
	on the incoming edges to $v \in V_n$ such that this definition makes $\tau$ continuous at these maximal paths.
	
	A graph covering map $f:\GC \to \GC$ on the weighted graph cover
	$$
	\GC = \IL(\GC_n,\pi_n)
	= \{ (\gamma_n)_{n \geq 0} : \pi_{n+1}(\gamma_{n+1}) = \gamma_n \in \GC_n \text{ for all } n \geq 0\}
	$$
	(where $\GC_0$ is a single loop from a single vertex)
	can be turned in a Bratteli-Vershik system as follows.
	The edges $a \in \GC_n$ correspond bijectively to the vertices $v= V_n$ of an ordered Bratteli diagram.
	We call the bijection $P_n$.
	If the bonding map $\pi_n$ maps $a$ to the concatenation $\pi_n(a) = a_1 \cdots a_k$
	of edges in $\GC_{n-1}$, then we draw $k$ incoming edges $e_i \in E_n$ to $v = P_n(a)$ from the vertices
	$P_{n-1}(a_i) = \sbv(e_i) \in V_{n-1}$, $i = 1,2, \dots, k$ in this order.
	The map $f:\GC \to \GC$ will then lift to the Vershik map $\tau$ on the path space $X_{BV}$.
	
	\begin{remark}
		This construction doesn't guarantee that the emerging Bratteli diagram
		allows a unique (or useful) definition of the Vershik map.
		For example, if we start with the graph cover of Figure~\ref{fig:gc_ful_shift}, then we obtain the full binary tree as Bratteli diagram. Every path in it is both minimal and maximal, and there is no sensible way of defining $\tau$ on it
		that represents in any way the one-sided shift.
	\end{remark}

	Conversely, when given an ordered Bratteli diagram $(E_n, V_n)_{n \geq 1}$,
	the vertices  $v \in V_n$ correspond bijectively to edges $a = P_n^{-1}(v)$ in $\GC_n$.
	The ordered set of incoming edges $e_i \in E_n$ to $v$ determine the bonding map $\pi_n$
	by the concatenation $\pi_n(a) = a_1 \dots a_k$ if $P_{n-1}(a_i) = \sbv(e_i)$.
	The Vershik map $\tau$ then translates to the graph covering map $f$.
	The more involved step of how to connect the edges of $\GC_n$, i.e., how to determine the vertices of $\GC_n$, follows from the following lemma.
	
	\begin{lemma}\label{lem:eq}
		Given $v,w \in V_n$, let $[v^{\min}]$ indicate the cylinder set given by the minimal path from $v$ to $v_0$, and $[w^{\max}]$ indicates the cylinder set given by the minimal path from $w$ to $v_0$.
		\begin{enumerate}
			\item 
			For $v,v' \in V_n$, if there is $w \in V_n$ such that
			\begin{equation}\label{eq:sbv}
			\tau([w^{\max}]) \cap [v^{\min}] \neq \void \neq \tau([w^{\max}]) \cap [{v'}^{\min}],
			\end{equation}
			then $\sbv(P^{-1}(v)) = \sbv(P^{-1}(v'))$.
			\item 
			For $w,w' \in V_n$, if there is $v \in V_n$ such that
			\begin{equation}\label{eq:tbv}
			\tau([w^{\max}]) \cap [v^{\min}] \neq \void \neq \tau([{w'}^{\max}]) \cap [v^{\min}],
			\end{equation}
			then $\tbv(P^{-1}(v)) = \tbv(P^{-1}(v'))$. 
		\end{enumerate}
	\end{lemma}
	
	\begin{proof}
		If $\tau([w^{\max}]) \cap [v^{\min}] \neq \void$, then there must be some $m > n$ and $u \in V_m$ such that there are paths $w \to u$ and successor path
		$v \to u$.
		Keeping in mind that vertices in $V_n$ correspond to edges in $\GC_n$,
		this translates into:
		There is $m > n$ and $\gamma \in \GC_m$ such that
		$P^{-1}(w) P^{-1}(v)$ is a subpath of $\pi_{n+1} \circ \cdots \circ \pi_m(\gamma)$ and in particular $\sbv(P^{-1}(v)) = \tbv(P^{-1}(w))$.
		
		If the same holds for $w$ and $v'$, so $P^{-1}(w) P^{-1}(v')$ is a subpath of $\pi_{n+1} \circ \cdots \circ \pi_{m'}(\gamma')$,
		then $\sbv(P^{-1}(v)) = \sbv(P^{-1}(v')) = \tbv(P^{-1}(w))$.
		
		This proves (i). The proof for (ii) is analogous.
	\end{proof}
	
	If we write $v \sim_{\sbv} v'$ if \eqref{eq:sbv} holds, the $\sim_{\sbv}$ is a reflexive, symmetric relation, but not necessarily transitive, as the next example shows.
	Therefore we need to take the transitive hull. Shimomura  \cite{Shi20}\footnote{although he fails to take the transitive hull.}
	calls the equivalence classes of $\sim$ {\bf clusters}.
	Edges $\gamma, \gamma' \in \GC_n$ have $\sbv(\gamma) = \sbv(\gamma')$ if and only they are in the same cluster.
	The analogous statement can be made as a necessary and sufficient 
	condition for $\tbv(\gamma) = \tbv(\gamma')$.
	
	\begin{example}
		Figure~\ref{fig:eq} shows the graph cover and Bratteli-Vershik representation of the stationary substitution shift based on the substitution
		$$
		\chi: \begin{cases}
		1 \to 12,\\
		2 \to 13,\\
		3 \to 123.
		\end{cases}
		$$
		Each graph $\GC_n$ has only one vertex, even though $1 \not\sim_{\sbv} 2$.
		This shows that taking the transitive hull of $\sim_{\sbv}$ is essential
		to get an equivalence relation.
		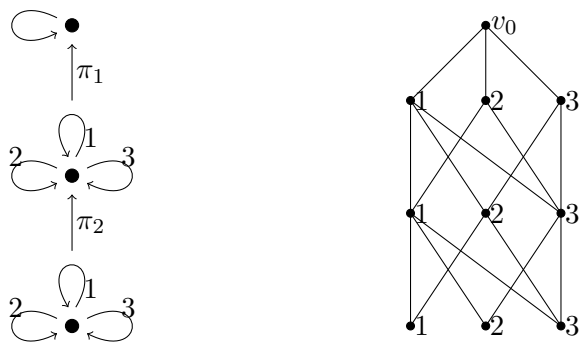
\begin{figure}[ht]
			\begin{center}
				\begin{tikzpicture}[scale=0.5]
				\filldraw(-1,7) circle (5pt);
				\draw[->, draw=black] (-1.4, 7.2) .. controls (-3, 8) and (-3, 6) .. (-1.4, 6.8);
				\draw[->, draw=black] (-1,5) -- (-1,6.5); \node at (-0.5,5.7) {\small $\pi_1$};
				\filldraw(-1,3) circle (5pt);
				\draw[->, draw=black] (-0.9, 3.5) .. controls (0, 5) and (-2, 5) .. (-1.1, 3.5);
				\node at (-0.5,4) {\small $1$};
				\draw[->, draw=black] (-1.4, 3.2) .. controls (-3, 4) and (-3, 2) .. (-1.4, 2.8);
				\node at (-2.5,3.5) {\small $2$};
				\draw[->, draw=black] (-0.6, 3.2) .. controls (1, 4) and (1, 2) .. (-0.6, 2.8);
				\node at (0.5,3.5) {\small $3$};
				\draw[->, draw=black] (-1,1) -- (-1,2.5); \node at (-0.5,1.7) {\small $\pi_2$};
				\filldraw(-1,-1) circle (5pt);
				\draw[->, draw=black] (-0.9, -0.5) .. controls (0, 1) and (-2, 1) .. (-1.1, -0.5);
				\node at (-0.5,0) {\small $1$};
				\draw[->, draw=black] (-1.4, -0.8) .. controls (-3, 0) and (-3, -2) .. (-1.4, -1.2);
				\node at (-2.5,-0.5) {\small $2$};
				\draw[->, draw=black] (-0.6, -0.8) .. controls (1, 0) and (1, -2) .. (-0.6, -1.2);
				\node at (0.5,-0.5) {\small $3$};
				\filldraw(10,7) circle (3pt); \node at (10.5,7) {\small $v_0$};
				\draw[-, draw=black] (10,7) -- (8,5);
				\draw[-, draw=black] (10,7) -- (10,5);
				\draw[-, draw=black] (10,7) -- (12,5);
				\filldraw(8,5) circle (3pt); \node at (8.3,5) {\small $1$};
				\filldraw(10,5) circle (3pt); \node at (10.3,5) {\small $2$};
				\filldraw(12,5) circle (3pt); \node at (12.3,5) {\small $3$};
				\draw[-, draw=black] (8,2) -- (8,5);
				\draw[-, draw=black] (8,2) -- (10,5);
				\draw[-, draw=black] (10,2) -- (12,5);
				\draw[-, draw=black] (10,2) -- (8,5);
				\draw[-, draw=black] (12,2) -- (10,5);
				\draw[-, draw=black] (12,2) -- (12,5);
				\draw[-, draw=black] (12,2) -- (8,5);
				\filldraw(8,2) circle (3pt); \node at (8.3,2) {\small $1$};
				\filldraw(10,2) circle (3pt); \node at (10.3,2) {\small $2$};
				\filldraw(12,2) circle (3pt); \node at (12.3,2) {\small $3$};
				\draw[-, draw=black] (8,-1) -- (8,2);
				\draw[-, draw=black] (8,-1) -- (10,2);
				\draw[-, draw=black] (10,-1) -- (12,2);
				\draw[-, draw=black] (10,-1) -- (8,2);
				\draw[-, draw=black] (12,-1) -- (10,2);
				\draw[-, draw=black] (12,-1) -- (12,2);
				\draw[-, draw=black] (12,-1) -- (8,2);
				\filldraw(8,-1) circle (3pt); \node at (8.3,-1) {\small $1$};
				\filldraw(10,-1) circle (3pt); \node at (10.3,-1) {\small $2$};
				\filldraw(12,-1) circle (3pt); \node at (12.3,-1) {\small $3$};
				\end{tikzpicture}
				\caption{The graph cover and Bratteli-Vershik representation of a substitution shift.}
				\label{fig:eq}
			\end{center}
		\end{figure}
	\end{example}
	
	\section{S-adic transformations}

	Instead of using a single substitution to create an infinite word $\rho \in \Alf^\N$,
	we can use a sequence of substitutions
	$\chi_n : \Alf_n \to \Alf_{n-1}^*$ (where $\Alf_{n-1}^*$ denotes the set of non-empty finite words in the alphabet $\Alf_{n-1}$), potentially between different alphabets $\Alf_n$.
	Thus
	\begin{equation}\label{eq:rho_adic}
	\rho = \lim_{n \to \infty} \chi_1 \circ \chi_2 \circ \dots \circ \chi_n(a_n), \qquad a_n \in \Alf_n.
	\end{equation}
	A priori, the limit need not exist, or can depend on the choice of letters $a_n \in \Alf_n$,
	but if $\rho$ exists and is an infinite sequence, then we have the following definition.
	Ferenczi \cite{Fer} was the first to call such systems S-adic,
	and they gave e.g.\ handy tool to describe rotations on the circle \cite{MH40}, \cite[Section 6.3]{fogg} and the torus \cite{AR}, and their renormalizations.

	\begin{definition}\label{def:S_adic}
		Let $S$ be a collection\footnote{Some, but not all, authors require $S$ to be finite.}
		of substitutions $\chi$ and choose $\chi_n \in S$
		such that alphabets match: $\chi_n:\Alf_n \to \Alf_{n-1}^*$.
		Assume that the sequence $\rho$ defined in \eqref{eq:rho_adic} exists and is infinite, and let
		$X_\rho = \overline{\orb_\sigma(\rho)}$. Then $(X_\rho, \sigma)$ is called an S-adic shift.
	\end{definition}

	For S-adic transformation given by substitutions $\chi_n:\Alf_n \to \Alf^*_{n-1}$ for $n \geq 1$, the graph cover is
	given by graphs $\GC_n$ consisting of a single vertex and $\# \Alf$ edges labeled by the letters of $\Alf$.
	The bonding maps $\pi_n = \chi_n$, i.e., the $\pi_n$-image of
	the loop $a \in \GC_n$ is the concatenation
	of loops in $\GC_{n-1}$ in the order given by $\chi_n(a)$.
	The winding matrix $W_n$ is exactly the associated matrix of the substitution $\chi_n$.
	
	Conversely, if all $\GC_n$'s consist of a single vertex with some loops,
	the substitutions $\chi_n$ can be read off the way $\pi_n$ acts on these loops. For different structures of directed graphs, it is not so easy
	(and most of the time impossible) to associate S-adic transformations to them.
	
	\subsection{Interval exchange transformations}
	An interval exchange transformation is based on a partition of $[0,1)$ into $d \in \N$ intervals
	$I_j = [q_{j-1}, q_j)$ for $0 = q_0 < q_1 < \dots < q_d=1$, with lengths $\lambda_j = q_j-q_{j-1}$, so
	$\lambda_1 + \dots + \lambda_d = 1$.
	It rearranges these intervals according to a permutation $\zeta$ of $\{ 1, \dots, d\}$, i.e.,
	$$
	T(x) = x-q_j+q_{\zeta(j)} \qquad \text{ if } x \in I_j.
	$$
	One can analyse $T$ using first return maps to
	subintervals $J$ that are the union of adjacent intervals, taken either from $\{ I_j\}_{j=1}^d$
	or from $\{ T(I_j) \}_{j=1}^d$.
	Keane \cite{K75} showed that in this case the first return map $T': J \to J$
	has at most $d$ intervals of continuity $I'_j$; in fact, if the so-called {\em Keane condition}
	``$T^n(q_j) \neq q_k$ for all $1 \leq j \leq k < d$ and $n \in \Z \setminus \{ 0 \}$'' holds, then $T'$ has exactly $d$ intervals $I'_j$ of continuity. Moreover, Keane's condition implies that $T$ is minimal.

	Symbolically, taking the first return map can be described using a substitution $\chi$ on the alphabet $\Alf = \{ 1, \dots, d\}$:
	\begin{eqnarray*}
	\chi(a) &=& a_0a_2 \dots a_{r(a)-1} \in \cA^*, \\
	T^j(I'_a) &\subset& I_{a_j} \text{ for }
	0 \leq j < r(a) := \inf\{ n \geq 1 : T^n(I_a) \subset J\}.
	\end{eqnarray*}
    Recursively, we can obtain a sequence of interval exchange transformations $T_0 = T:J_0 \to J_0 := [0,1)$
    and $T_n:J_n \to J_n \subset J_{n-1}$, with substitutions $\zeta_n$, and having $\{I^n_j\}_{j=1}^d$ as intervals of continuity. We have corresponding substitutions
    $(\chi_n)_{n \geq 1}$ on $\cA$ and the infinite sequence $\rho$
    as in \eqref{eq:rho_adic}, where the symbols $a_n$ are such that $I^n_{a_n} \subset J_{n+1}$ for all $n \geq 1$,
    expresses the itinerary of $x \in \bigcap J_n$
    w.r.t.\ the original partition $\{ I_j\}$.

    As such, in the corresponding weighted graph cover, every graph $\GC_n$ is a single vertex with $d$ loops, and $\pi_n$ wraps
    $\gamma_a \in \GC_n$ around $\GC_{n-1}$ in the order prescribed by the substitution $\chi_n$.

\begin{example}
    A special case is the Rauzy induction, see \cite{Via06} or \cite[Section 4.4]{B23}. In this case,
    \begin{eqnarray*}
\begin{array}{l}
 \text{\bf Type 0:} \\[3mm]
J_n = J_{n-1} \setminus  T_{n-1}(I^{n-1}_e), 
\end{array}  & \chi_n : \begin{cases}
            a \mapsto a & a \leq e\\
            (e+1) \to ed \\
            a \mapsto (a-1) & a > e+1 \
           \end{cases}
\text{ if } |I^{n-1}_d| > |I^{n-1}_e|, \\[2mm]
\begin{array}{l}
 \text{\bf Type 1:} \\[3mm] J_n = J_{n-1} \setminus  I^{n-1}_d, \end{array} \qquad \ \ &
  \chi_n : \begin{cases}
            a \mapsto a \qquad\quad & a \neq e \qquad   \\
            e \mapsto ed
           \end{cases}
\text{ if } |I^{n-1}_d| < |I^{n-1}_e|,   \\
    \end{eqnarray*}
where $e = \zeta_{n-1}^{-1}(d)$ is such that the interval $T_n(I^{n-1}_e)$ is adjacent to the right endpoint of $J_n$.
\end{example}

The other direction, i.e., recognizing which graph covers represent interval exchange transformations, is complicated,
because which sequences of substitutions
$(\chi_n)_{n \geq 1}$ (and of permutations $(\zeta_n)_{n \geq 1}$ ) are allowed is not easy to track.
If Rauzy induction is used, then at least the allowed sequences
of permutations are given by paths in the so-called Rauzy classes
\cite[Section 6]{Via06}, but Rauzy induction provides
only one class of representations.
For example, Gjerde \& Johansen \cite{GJ02} used a representation in which $J_n = J_{n-1} \setminus I^{n-1}_d$ for all $n \geq 1$.

\subsection*{Acknowledgements}
We gratefully acknowledge the support of the Austrian Exchange Service (WTZ Project PL 15/2022) and co-financing of the Polish National Agency for Academic Exchange under contract no. PPN/BAT/2021/1/00024/U/00001.

	\bigskip\bigskip
\noindent
J. Boro\'nski\\	
Department of Differential Equations\\
Faculty of Mathematics and Computer Science\\
Jagiellonian University\\
ul. {\L}ojasiewicza 6, 30-348 Krak\'ow, Poland\\
e-mail: jan.boronski@uj.edu.pl\\
\smallskip

\noindent
H. Bruin\\
Faculty of Mathematics\\ 
University of Vienna\\
Oskar Morgensternplatz 1\\
Vienna, Austria\\
e-mail: henk.bruin@univie.ac.at\\
\smallskip

\noindent
P. Kucharski\\	
Department of Differential Equations\\
Faculty of Mathematics and Computer Science\\
Jagiellonian University\\
ul. {\L}ojasiewicza 6, 30-348 Krak\'ow, Poland\\
e-mail: przemyslaw.kucharski@doctoral.uj.edu.pl
\end{document}